\documentclass[smallextended]{svjour3}
\smartqed
\usepackage{amsmath}
\usepackage{graphicx}
\usepackage{float}
\usepackage{pst-grad}
\usepackage{pst-plot} 
\usepackage{pstricks-add}

\spnewtheorem{theorema}{Theorem}{\bfseries}{\itshape}

\title
{
	On uniquely $k$-list colorable planar graphs, graphs on surfaces, and regular graphs
}
\author
{
	M. Abdolmaleki \and J. P. Hutchinson \and \\ S. Gh. Ilchi \and E. S. Mahmoodian \and \\ M. A. Shabani
}
\institute
{
	M. Abdolmaleki \and E. S. Mahmoodian \and M. A. Shabani
	\at
		Department of Mathematical Sciences, Sharif University of Technology, P.O. Box 11155-9415 Tehran, I. R. Iran\\
		\email{emahmood@sharif.edu}\\
		\email{\{mojtabaabdolmaleki2009,aminshabaany\}@gmail.com}
	\and
	J. P. Hutchinson
	\at
		Department of Mathematics, Statistics, and Computer Science, Macalester College, St. Paul, MN, USA\\
		\email{hutchinson@macalester.edu}
	\and
	S. Gh. Ilchi
	\at
		Department of Computer Engineering, Sharif University of Technology, P.O. Box 11155-9517, Tehran, I. R. Iran\\
		\email{silchi72@gmail.com}
}
\begin{document}
\maketitle

\begin{abstract}
A graph $G$ is called {\sf uniquely\ $k$-list\ colorable\ } (U$k$LC) if there exists a list of colors on its vertices, say $L=\lbrace S_v \mid v \in V(G) \rbrace $,
each of size $k$, such that there is a unique proper list coloring of $G$ from this list of colors. A graph $G$ is said to have {\sf property $M(k)$} if it is not uniquely $k$-list colorable. 
Mahmoodian and Mahdian~\cite{MR1675193} characterized all graphs with property $M(2)$. For $k\geq 3$ property $M(k)$ has been studied only for multipartite graphs. 
Here we find bounds on $M(k)$ for graphs embedded on surfaces, and obtain new results on planar graphs. We begin a general study of bounds on $M(k)$ for regular graphs, as well as for graphs with varying list sizes. 
\end{abstract}

\section {Introduction and preliminary results}
\label{intro-sec}
Let $G=(V,E)$ be a graph with a set of colors available on each vertex $v\in V(G)$, say $S_v$. A function $c: V(G) \rightarrow \bigcup_{v\in V(G)}^{} S_{v}$ is called a {\sf proper list coloring} of $G$ if it has the following conditions:
\begin{equation*}
\forall v\in V(G) \Rightarrow c(v) \in S_{v}
\end{equation*}
\begin{equation*}
\forall (u,v)\in E(G) \Rightarrow c(u) \neq c(v).
\end{equation*}

$G$ is called a \begin{sf}uniquely $L$-list colorable graph\end{sf} if there exists a collection of sets $L=\lbrace S_v \mid v \in V(G) \rbrace $ from which we have exactly one proper list  coloring of $G$; $G$ is also called a \begin{sf}uniquely $k$-list colorable graph (U$k$LC)\end{sf}, if each list has a size of $k$. In the opposite sense, a graph has {\sf property $M(k)$}, M for Marshall Hall, if and only if it is not uniquely $k$-list colorable. For example, every graph is U1LC and so does not have property $M(1)$. UkLC was independently introduced by Mahmoodian and Mahdian~\cite{MR1625324} and by Dinitz and Martin~\cite{dinitz1995stipulation}.

Mahdian and Mahmoodian~\cite{MR1675193} proved the following theorem.

\begin{theorema}{\rm (\cite{MR1675193})}
\label{m2}
A connected graph has property $M(2)$ if and only if every block of the graph is either a cycle, a complete graph, or a complete bipartite graph.
\end{theorema}

Also complete multipartite U3LC graphs are completely 
characterized in~\cite{MR1832463,MR2429150,MR2239316,MR2364772}.
All but finitely many complete multipartite graphs with at least six parts that are U4LC are characterized in  
\cite{MR2568835}.
In~\cite{MR2431825} it is shown that recognizing uniquely $k$-list colorable graphs is $\sum_{2}^{p}$-complete for every $k \geq 3$.
As $k$ increases, it gets more difficult to characterize U$k$LC complete multipartite graphs 
\cite{MR3115283}.\\

In~\cite{MR1906860} $\chi_u(G)$ is defined as follows:

\begin{definition}
A U$k$LC graph $G$ is {\sf $(k,t)$-list colorable} if it has a list color assignment on vertices of $G$ such that the number of different colors in the union of all lists is at most $t$.
For a graph $G$ and a positive integer $k$, we define 
\begin{sf}
$\chi_u(G,k)$
\end{sf}
to be the minimum number $t$ such that $G$ is uniquely $(k,t)$-list colorable and 0 if $G$ is not uniquely $k$-list colorable.
\begin{sf} The uniquely list chromatic number of a graph $G$,
\end {sf} denoted by \begin{sf}
$\chi_u(G),$
\end{sf} is defined to be $\max_{k\geq1}$ $\chi_u(G,k)$.
\end{definition}


The following theorem on $\chi_u(G,2)$ is also given in~\cite{MR1906860}.

\begin{theorema}{\rm (\cite{MR1906860})}
A graph $G$ is uniquely 2-list colorable if and only if it is uniquely $(2, t)$-list colorable, where $t = \max(3,\chi(G))$.
\label{chi_theorem}
\end{theorema}

Later we show that for $k\leq3$ every $k$-regular connected graph $G$ has $\chi_u(G) \leq \Delta + 1$ = $k+1$. This proves a conjecture in~\cite{MR1906860} in some cases.

Let $m(G)$ be the least $k$ for which a graph has property $M(k)$. Equivalently $k=m(G)-1$ is the largest integer $k$ such that $G$ is U$k$LC. It can be proved easily that for every graph $G$, $2\leq m(G)\leq \min(\delta+2,n-1)$ where $\delta$ is the minimum degree of $G$ with $n$ vertices.


\begin{theorema}{\rm (\cite{MR1906860})}
\label{induced}
Suppose $H$ is an induced subgraph of a graph $G$ with the following property:
\begin{itemize}
\item $H$ has property $M(k)$, and
\item each vertex of $H$ is adjacent to at most $l$ vertices of $V(G)\setminus V(H)$.
\end{itemize}
Then $G$ has property $M(k+l)$.
\end{theorema}
In the later sections, we  obtain a lower bound on the average degree of all U$k$LC graphs. Based on this lower bound, we will state some results about planar graphs, graphs on surfaces and regular graphs. Additionally in the last section, we  study the graphs where the lists of colors may have different sizes.

\section {Main results}
\label{main-sec}
Consider a graph $G = (V, E)$ with a given $k$-list assignment $L$, and let $c$ be an $L$-list coloring of $G$. We call the following procedure on $G$ a
\begin{sf}
(G,c)-directing procedure.
\end{sf}
\begin{enumerate}
\item{
For each edge $(u,v)\in E(G)$, we give a direction from $v$ to $u$ if the following two conditions are satisfied: 
\begin{itemize}
\item
$u$ is the only neighbor of $v$ with color $c(u)$, and
\item
$c(u) \in L(v)$.
\end{itemize}
}
 
 \item Denote the (partially directed) graph that is obtained so far in the directing procedure as $G'$. For each pair \{$v$,$t$\} with $v\in V(G')$ and $t\in L(v)\setminus \left\{c(v) \right\}$, consider $A = \left\{ a_1, a_2, \ldots, a_r \right\}$,  the set of all the neighbors of $v$ with $c\left( a_i \right) = t$. For each such pair    
 \{$v$,$t$\} for which exactly one of the edges between $v$ and the vertices of $A$ is undirected, say $e=(v, a_i)$, and all other edges are directed to $v$, then simultaneously direct each of these edges $e$ from $v$ to $a_i$.  
 
If, after giving all such directions, a vertex and a color appear having such property,  
then we go to Step 2 again. Otherwise terminate the procedure.
\end{enumerate}
We denote the graph $D(G,c,0) = G$. Denote the graph that is obtained at the end of Step $1$ by $D(G,c,1)$, and for each $i\geq 2$, denote $D(G,c,i)$ the graph obtained after $(i-1)$-th iteration of Step $2$. Also, we denote the final (partially or totally directed) graph by $D(G,c)$.
\begin{definition}
A $(t_1,t_2)$-color alternating path in a graph is a path that is properly list-colored with colors $t_1$ and $t_2$. A directed $(t_1,t_2)$-color alternating path in a partially (or totally) directed graph is a $(t_1,t_2)$-color alternating path with all edges directed coherently (all in one direction).
\end{definition}

\begin{lemma}
\label{amin}
If $G$ is U$k$LC, then there is no bidirectional edge in $D(G,c)$.
\end{lemma}
\begin{proof}
Note that if in Step 1 one edge is bidirectional, then colors on the endpoints can be switched and G is not UkLC. Otherwise, suppose that $i>1$ is the least number such that $D(G,c,i)$ has a bidirectional edge $e=(u,v)$ and let $G' = D(G,c,i-1)$. Consider $S$, the set of all vertices in $G'$ that have a directed $(c(u),c(v))$-color alternating path to either $u$ or $v$. Now we can switch colors of vertices in $S$ between $c(u)$ and $c(v)$ and get a new coloring $c'$, as explained below.

For each vertex $x\in S$, if the edge $e'$ of $x$ in the color alternating path from $x$ to $u$ or $v$ is directed in Step 1 of the directing procedure, then $x$ has no other adjacent vertices with colors $c(u)$ or $c(v)$. Otherwise, if the edge $e'$ of $x$ is directed in  Step 2 in the directing procedure, then all neighbors of $x$ with color $c(u)$ or $c(v)$ must have a directed edge to $x$ except $e'$ and so they also must be in $S$. So $c'$ is another proper coloring of $G$ and we can conclude that $G$ has property $M(k)$.
\qed
\end{proof}

\begin{lemma}
\label{lmleast}
If $G$ is a U$k$LC graph with a unique proper $k$-list coloring $c$, then for each vertex $v \in D(G,c)$ and color $t\in L(v)\backslash \{c(v)\}$, there is at least one undirected edge or an out-directed edge from $v$ to one of the neighbors of $v$ with color $t$.
\end{lemma}
\begin{proof}
Suppose there is a $v \in G$ and color $t\in L(v)\backslash c(v)$ such that all neighbors of $v$ with color $t$ have a directed edge to $v$ in $D(G,c)$. Consider $S \subset V(D(G,c))$ as the set of all vertices with color $t$ or $c(v)$ that have a directed $(c(v),t)$-color alternating path to $v$. Now because each vertex in $S\backslash {v}$ has an out-directed edge, the list of each vertex contains $t$ and $c(v)$ and there is no other vertex with color $t$ or $c(v)$ in $G\setminus\{S\}$ that is adjacent to a vertex in $S$. We can switch colors of vertices in $S$ between $t$ and $c(v)$ and get a new proper coloring.
\qed
\end{proof}

\begin{definition}
Define $d^+(v,t)$ as the number of directed edges from vertex $v$ to a vertex with color $t$ in $D(G,c)$ and $d^-(v,t)$ as the number of directed edges from a vertex with color $t$ to $v$. Also, $d^+(v)=\sum_{t\in L(v)} d^+(v,t)$ and similarly $d^-(v)=\sum_{t\in L(v)} d^-(v,t)$.
\end{definition}
The next result appeared in  {\rm~\cite{MR1899929}}; here we give an
equivalent formulation with alternative short proof  which also may be used to prove Theorem~\ref{main_result}.

\begin{theorem}
\label{main_result}
All graphs with average degree less than $2k-2$ have property $M(k)$.
\end{theorem}
\begin{proof}
Suppose that there is a U$k$LC graph $G$ with average degree less than $2k-2$; let $d(v)$ denote the degree of vertex $v$ in $V(G)$. Then there exists a proper coloring $c$ on $G$, and let  $G^* = D(G, c)$.
For each vertex $v\in V(G^*)$, set $L'(v)=L(v)\setminus\{c(v)\}$. We introduce an upper bound on $d^-(v, t) - d^+(v, t)$.  Let $N(v,t)$ be the number of vertices adjacent to $v$ with color $t$. Since $G$ is U$k$LC, it is clear that $N(v,t) \geq 1$.
Now there are two situations: 

\begin{itemize}
\item $N(v,t) = 1$: Obviously this edge is directed out of the $v$ in Step 1 of the directing procedure.
\item $N(v,t) > 1$: 
If $d^-(v, t) \leq N(v,t)-2$, then $d^-(v, t) - d^+(v, t) \leq N(v,t)-2$.
Otherwise according to Lemma 2 $d^-(v, t) = N(v,t)-1$ and by Step 2 of the directing procedure $d^+(v, t) = 1$.
\end{itemize}

Thus in all cases $d^-(v, t) - d^+(v, t)$ is at most $N(v,t) - 2$,   
\begin{equation*}
d^-(v) - d^+(v) \leq \left [ \sum_{t\in L'(v)} (d^-(v, t) - d^+(v, t)) \right ] + d(v) - \sum_{t\in L'(v)} N(v,t) \leq d(v) - 2(k-1). 
\end{equation*}

Now if we sum up these inequalities over all vertices of $G^*$, then we can conclude that:
\begin{equation*} 
\sum_{v\in V(G)} d^-(v) - d^+(v) \leq 2|E(G)| - |V(G)|\times (2k-2) < 0.
\end{equation*}

The above inequality contradicts the fact that the sum of the in-degrees of all vertices is equal to the sum of the out-degrees of all vertices, and the proof is complete. 
\qed
\end{proof}

Also there exist infinitely many UkLC graphs such that their average degrees tend to $2k-2$. So the theorem above gives a tight inequality for the average degree of UkLC graphs.

\begin{lemma}
\label{tight}
The inequality in Theorem~\ref{main_result} is tight.
\end{lemma}
\begin{proof}
For $k\geq2, n\geq2$, let $T(k,n)$ be the graph with $kn$ vertices where vertices $j$ and $j'$ $(1\leq j < j' \leq kn)$ are adjacent if and only if $j'-j < k$ or $j+kn-j' < k$. 
For $1\leq i \leq k+1$ set $G_i = T(k,n)$ and assign to each vertex of $G_i$ the set of colors $\{1,2,\cdots,k+1\} \setminus \{i\}$. It's obvious that $k$ consecutively-labeled vertices in $G_i$ get different colors in a proper list-coloring. Then for each $i$ and $j$ with  $1\leq i\leq k+1$ and $1\leq j \leq k-1$,  add an edge between vertex $j$ of $G_i$ and vertices $1,2,\cdots,k$ of the graph $G_{(i+j) \pmod{(k+1)}}$. In a list coloring of the resulting graph, vertex $j$ of graph $G_i$ must be colored with $(i+j) \pmod{(k+1)}$. Since there is a unique coloring of vertices $1,2,\cdots,k-1$ of each $G_i$, there is a unique coloring of the rest of the graph.  Thus the proposed graph is U$k$LC, and the average degree of this graph equals to:
\begin{equation*}
\frac{(k+1)[nk(2k-2) + (k-1)k]}{(k+1)kn} = 2k-2 + \frac{k-1}{n}.
\end{equation*}
The latter expression converges to $2k-2$  when $n \rightarrow \infty$. 
\qed
\end{proof}

By Theorem~\ref{main_result} we can conclude the following statements about planar graphs.

The following two corollaries were stated in {\rm~\cite{MR1899929}}.
\begin{corollary}
\label{cor1}
Each planar graph has property $M(4)$.
\end{corollary}
\begin{proof}
By Euler's formula, for any planar graph $G(V, E)$, $\left | E(G)\right | \leq 3\left | V(G)\right | - 6$. So the average degree of any planar graph is less than $6$. Hence all  planar graphs have property $M(4)$.
\qed
\end{proof}
\begin{corollary}
Every triangle-free planar graph has property $M(3)$.
\end{corollary}
\begin{proof}
Again according to Euler's formula, the average degree of any triangle-free planar graph is less than $4$, so all of them have property $M(3)$.
\qed
\end{proof}
\begin{corollary}
Every outer-planar graph has property $M(3)$.
\end{corollary}
\begin{proof}
Let $G$ be a connected outer-planar graph. If $G$ has $|F(G)|$ faces, 
then $2|E(G)| \geq 3(F(G)-1)+|V(G)|$. Substituting the resulting upper bound for $|F(G)|$ into Euler's formula leads to the upper bound on  average degree: $2\frac{|E(G)|}{|V(G)|}<4$, so $G$ has property $M(3)$.
\qed
\end{proof}
\begin{remark}
There are infinitely many U3LC planar graphs.
\end{remark}

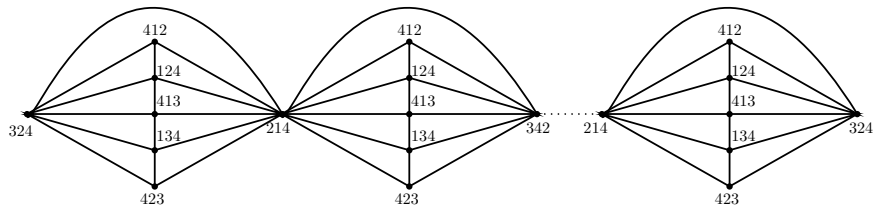
\begin{figure}[H]
\psscalebox{0.6 0.6} 
{ 
\begin{pspicture}(0,-2.2775712)(19.08,2.2775712)
\psdots[linecolor=black, dotsize=0.14](8.78,1.3783664)
\psdots[linecolor=black, dotsize=0.14](8.78,0.5783663)
\psdots[linecolor=black, dotsize=0.14](8.78,-0.22163369)
\psdots[linecolor=black, dotsize=0.14](8.78,-1.0216337)
\psdots[linecolor=black, dotsize=0.14](8.78,-1.8216337)
\psdots[linecolor=black, dotsize=0.14](5.98,-0.22163369)
\psdots[linecolor=black, dotsize=0.14](11.58,-0.22163369)
\psline[linecolor=black, linewidth=0.04](5.98,-0.22163369)(8.78,1.3783664)(11.58,-0.22163369)(8.78,0.5783663)(5.98,-0.22163369)(11.58,-0.22163369)(8.78,-1.0216337)(5.98,-0.22163369)(8.78,-1.8216337)(11.58,-0.22163369)
\psline[linecolor=black, linewidth=0.04](8.78,1.3783664)(8.78,-1.8216337)
\psbezier[linecolor=black, linewidth=0.04](5.98,-0.22163369)(5.98,-1.0216337)(8.08,5.4383664)(11.58,-0.22163369178771972)
\rput[bl](8.52,1.5183663){\fontsize{10 pt}{10 pt}\selectfont{$412$}}
\rput[bl](8.82,0.6183663){\fontsize{10 pt}{10 pt}\selectfont{$124$}}
\rput[bl](8.82,-0.04163369){\fontsize{10 pt}{10 pt}\selectfont{$413$}}
\rput[bl](8.82,-0.8216337){\fontsize{10 pt}{10 pt}\selectfont{$134$}}
\rput[bl](8.46,-2.2016337){\fontsize{10 pt}{10 pt}\selectfont{$423$}}
\psdots[linecolor=black, dotsize=0.14](3.2,1.3783664)
\psdots[linecolor=black, dotsize=0.14](3.2,0.5783663)
\psdots[linecolor=black, dotsize=0.14](3.2,-0.22163369)
\psdots[linecolor=black, dotsize=0.14](3.2,-1.0216337)
\psdots[linecolor=black, dotsize=0.14](3.2,-1.8216337)
\psdots[linecolor=black, dotsize=0.14](0.4,-0.22163369)
\psdots[linecolor=black, dotsize=0.14](6.0,-0.22163369)
\psline[linecolor=black, linewidth=0.04](0.4,-0.22163369)(3.2,1.3783664)(6.0,-0.22163369)(3.2,0.5783663)(0.4,-0.22163369)(6.0,-0.22163369)(3.2,-1.0216337)(0.4,-0.22163369)(3.2,-1.8216337)(6.0,-0.22163369)
\psline[linecolor=black, linewidth=0.04](3.2,1.3783664)(3.2,-1.8216337)
\psbezier[linecolor=black, linewidth=0.04](0.4,-0.22163369)(0.4,-1.0216337)(2.5,5.4383664)(6.0,-0.22163369178771972)
\rput[bl](2.94,1.5183663){\fontsize{10 pt}{10 pt}\selectfont{$412$}}
\rput[bl](3.24,0.6183663){\fontsize{10 pt}{10 pt}\selectfont{$124$}}
\rput[bl](3.24,-0.04163369){\fontsize{10 pt}{10 pt}\selectfont{$413$}}
\rput[bl](3.24,-0.8216337){\fontsize{10 pt}{10 pt}\selectfont{$134$}}
\rput[bl](2.88,-2.2016337){\fontsize{10 pt}{10 pt}\selectfont{$423$}}
\psdots[linecolor=black, dotsize=0.14](15.8,1.3783664)
\psdots[linecolor=black, dotsize=0.14](15.8,0.5783663)
\psdots[linecolor=black, dotsize=0.14](15.8,-0.22163369)
\psdots[linecolor=black, dotsize=0.14](15.8,-1.0216337)
\psdots[linecolor=black, dotsize=0.14](15.8,-1.8216337)
\psdots[linecolor=black, dotsize=0.14](13.0,-0.22163369)
\psdots[linecolor=black, dotsize=0.14](18.6,-0.22163369)
\psline[linecolor=black, linewidth=0.04](13.0,-0.22163369)(15.8,1.3783664)(18.6,-0.22163369)(15.8,0.5783663)(13.0,-0.22163369)(18.6,-0.22163369)(15.8,-1.0216337)(13.0,-0.22163369)(15.8,-1.8216337)(18.6,-0.22163369)
\psline[linecolor=black, linewidth=0.04](15.8,1.3783664)(15.8,-1.8216337)
\psbezier[linecolor=black, linewidth=0.04](13.0,-0.22163369)(13.0,-1.0216337)(15.1,5.4383664)(18.6,-0.22163369178771972)
\rput[bl](15.54,1.5183663){\fontsize{10 pt}{10 pt}\selectfont{$412$}}
\rput[bl](15.84,0.6183663){\fontsize{10 pt}{10 pt}\selectfont{$124$}}
\rput[bl](15.84,-0.04163369){\fontsize{10 pt}{10 pt}\selectfont{$413$}}
\rput[bl](15.84,-0.8216337){\fontsize{10 pt}{10 pt}\selectfont{$134$}}
\rput[bl](15.48,-2.2016337){\fontsize{10 pt}{10 pt}\selectfont{$423$}}
\psline[linecolor=black, linewidth=0.04, linestyle=dotted, dotsep=0.10583334cm](11.58,-0.22163369)(12.96,-0.22163369)
\rput[bl](0.0,-0.7016337){\fontsize{10 pt}{10 pt}\selectfont{$324$}}
\rput[bl](5.64,-0.6416337){\fontsize{10 pt}{10 pt}\selectfont{$214$}}
\rput[bl](11.34,-0.6416337){\fontsize{10 pt}{10 pt}\selectfont{$342$}}
\rput[bl](12.6,-0.6416337){\fontsize{10 pt}{10 pt}\selectfont{$214$}}
\rput[bl](18.42,-0.6416337){\fontsize{10 pt}{10 pt}\selectfont{$324$}}
\end{pspicture}
}

\caption{A family of U$3$LC planar graphs}
\label{planar1-fig}
\centering
\end{figure}

We also introduce a U3LC planer graph that does not contain any $K_4$ as a subgraph. It is a counterexample for a Conjecture in {\rm~\cite{MR1899929}}.

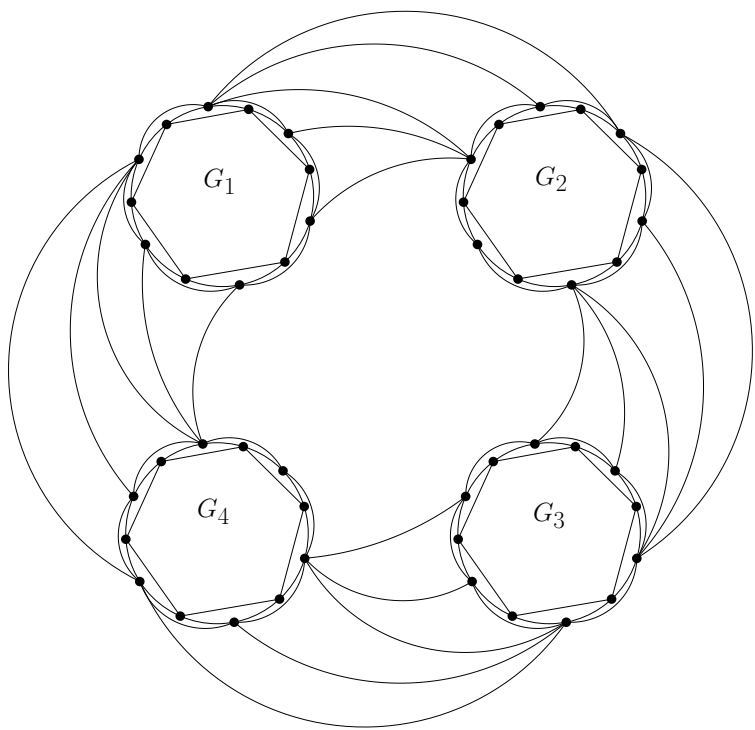
\begin{figure}[H]

\psscalebox{0.4 0.4} 
{
\newrgbcolor{xdxdff}{0 0 0}
\psset{xunit=1.0cm,yunit=1.0cm,algebraic=true,dimen=middle,dotstyle=o,dotsize=9pt 0,linewidth=0.8pt,arrowsize=3pt 2,arrowinset=0.25}
\begin{pspicture*}(0,-14)(25,10)
\pscircle(7.24,3.8){3.}
\psline(5.408985574061948,6.176422978345132)(8.107347062658953,6.6718824963596415)
\psline(8.107347062658953,6.6718824963596415)(10.107337026165851,4.682257546512569)
\psline(10.107337026165851,4.682257546512569)(9.294329756949,1.613740809118493)
\psline(9.294329756949,1.613740809118493)(6.035326486199539,1.0524990036129829)
\psline(6.035326486199539,1.0524990036129829)(4.246557098240255,3.601758748227831)
\psline(4.246557098240255,3.601758748227831)(5.408985574061948,6.176422978345132)
\parametricplot{1.2879484109342172}{3.162503955679627}{1.*1.7741374967508616*cos(t)+0.*1.7741374967508616*sin(t)+6.273943041591789|0.*1.7741374967508616*cos(t)+1.*1.7741374967508616*sin(t)+5.059168816449077}
\parametricplot{0.46203404514034274}{2.0280901344134667}{1.*1.9764720411620749*cos(t)+0.*1.9764720411620749*sin(t)+7.641744874377673|0.*1.9764720411620749*cos(t)+1.*1.9764720411620749*sin(t)+4.9894193083130025}
\parametricplot{-0.5451639760458011}{1.0267619156747703}{1.*2.110230768562858*cos(t)+0.*2.110230768562858*sin(t)+8.318740270402001|0.*2.110230768562858*cos(t)+1.*2.110230768562858*sin(t)+4.064899436637471}
\parametricplot{4.6414099533138975}{6.263799006984064}{1.*2.1641823373559164*cos(t)+0.*2.1641823373559164*sin(t)+7.95930111012953|0.*2.1641823373559164*cos(t)+1.*2.1641823373559164*sin(t)+3.0125746089058616}
\parametricplot{3.4612084584171816}{5.146713400951537}{1.*2.2631563520937377*cos(t)+0.*2.2631563520937377*sin(t)+6.85348764644498|0.*2.2631563520937377*cos(t)+1.*2.2631563520937377*sin(t)+2.9068741229370483}
\parametricplot{2.407330004843078}{4.020494328096271}{1.*1.9625913324215556*cos(t)+0.*1.9625913324215556*sin(t)+5.957074060812201|0.*1.9625913324215556*cos(t)+1.*1.9625913324215556*sin(t)+3.707057330970813}
\pscircle(18.16,3.8){3.}
\psline(16.32898557406195,6.176422978345132)(19.027347062658954,6.671882496359641)
\psline(19.027347062658954,6.671882496359641)(21.02733702616585,4.682257546512569)
\psline(21.02733702616585,4.682257546512569)(20.214329756949002,1.6137408091184948)
\psline(20.214329756949002,1.6137408091184948)(16.95532648619954,1.0524990036129829)
\psline(16.95532648619954,1.0524990036129829)(15.166557098240254,3.601758748227831)
\psline(15.166557098240254,3.601758748227831)(16.32898557406195,6.176422978345132)
\parametricplot{1.2879484109342156}{3.162503955679627}{1.*1.7741374967508599*cos(t)+0.*1.7741374967508599*sin(t)+17.193943041591787|0.*1.7741374967508599*cos(t)+1.*1.7741374967508599*sin(t)+5.059168816449077}
\parametricplot{0.4620340451403446}{2.028090134413464}{1.*1.9764720411620758*cos(t)+0.*1.9764720411620758*sin(t)+18.56174487437767|0.*1.9764720411620758*cos(t)+1.*1.9764720411620758*sin(t)+4.989419308313}
\parametricplot{-0.5451639760458011}{1.026761915674773}{1.*2.110230768562856*cos(t)+0.*2.110230768562856*sin(t)+19.238740270402005|0.*2.110230768562856*cos(t)+1.*2.110230768562856*sin(t)+4.064899436637471}
\parametricplot{4.641409953313896}{6.263799006984063}{1.*2.1641823373559164*cos(t)+0.*2.1641823373559164*sin(t)+18.879301110129532|0.*2.1641823373559164*cos(t)+1.*2.1641823373559164*sin(t)+3.0125746089058625}
\parametricplot{3.461208458417182}{5.146713400951539}{1.*2.2631563520937386*cos(t)+0.*2.2631563520937386*sin(t)+17.773487646444973|0.*2.2631563520937386*cos(t)+1.*2.2631563520937386*sin(t)+2.906874122937047}
\parametricplot{2.4073300048430784}{4.020494328096271}{1.*1.9625913324215567*cos(t)+0.*1.9625913324215567*sin(t)+16.877074060812202|0.*1.9625913324215567*cos(t)+1.*1.9625913324215567*sin(t)+3.7070573309708132}
\pscircle(7.063023843730591,-7.377705472590607){3.}
\psline(5.232009417792538,-5.001282494245476)(7.930370906389543,-4.505822976230966)
\psline(7.930370906389543,-4.505822976230966)(9.930360869896441,-6.495447926078038)
\psline(9.930360869896441,-6.495447926078038)(9.117353600679591,-9.563964663472113)
\psline(9.117353600679591,-9.563964663472113)(5.858350329930129,-10.125206468977623)
\psline(5.858350329930129,-10.125206468977623)(4.0695809419708455,-7.575946724362776)
\psline(4.0695809419708455,-7.575946724362776)(5.232009417792538,-5.001282494245476)
\parametricplot{1.287948410934216}{3.1625039556796275}{1.*1.7741374967508599*cos(t)+0.*1.7741374967508599*sin(t)+6.096966885322377|0.*1.7741374967508599*cos(t)+1.*1.7741374967508599*sin(t)+-6.118536656141529}
\parametricplot{0.4620340451403436}{2.028090134413466}{1.*1.9764720411620769*cos(t)+0.*1.9764720411620769*sin(t)+7.464768718108263|0.*1.9764720411620769*cos(t)+1.*1.9764720411620769*sin(t)+-6.188286164277606}
\parametricplot{-0.5451639760457985}{1.026761915674769}{1.*2.1102307685628623*cos(t)+0.*2.1102307685628623*sin(t)+8.141764114132586|0.*2.1102307685628623*cos(t)+1.*2.1102307685628623*sin(t)+-7.112806035953139}
\parametricplot{4.6414099533139}{6.263799006984061}{1.*2.1641823373559235*cos(t)+0.*2.1641823373559235*sin(t)+7.782324953860115|0.*2.1641823373559235*cos(t)+1.*2.1641823373559235*sin(t)+-8.16513086368474}
\parametricplot{3.461208458417178}{5.14671340095154}{1.*2.2631563520937323*cos(t)+0.*2.2631563520937323*sin(t)+6.676511490175566|0.*2.2631563520937323*cos(t)+1.*2.2631563520937323*sin(t)+-8.270831349653568}
\parametricplot{2.407330004843077}{4.020494328096271}{1.*1.9625913324215538*cos(t)+0.*1.9625913324215538*sin(t)+5.78009790454279|0.*1.9625913324215538*cos(t)+1.*1.9625913324215538*sin(t)+-7.470648141619795}
\pscircle(17.98302384373059,-7.377705472590607){3.}
\psline(16.152009417792538,-5.001282494245475)(18.850370906389543,-4.5058229762309665)
\psline(18.850370906389543,-4.5058229762309665)(20.850360869896438,-6.495447926078038)
\psline(20.850360869896438,-6.495447926078038)(20.03735360067959,-9.563964663472113)
\psline(20.03735360067959,-9.563964663472113)(16.778350329930127,-10.125206468977623)
\psline(16.778350329930127,-10.125206468977623)(14.989580941970843,-7.575946724362776)
\psline(14.989580941970843,-7.575946724362776)(16.152009417792538,-5.001282494245475)
\parametricplot{1.2879484109342156}{3.1625039556796275}{1.*1.7741374967508599*cos(t)+0.*1.7741374967508599*sin(t)+17.016966885322375|0.*1.7741374967508599*cos(t)+1.*1.7741374967508599*sin(t)+-6.118536656141529}
\parametricplot{0.46203404514034335}{2.028090134413465}{1.*1.9764720411620742*cos(t)+0.*1.9764720411620742*sin(t)+18.38476871810826|0.*1.9764720411620742*cos(t)+1.*1.9764720411620742*sin(t)+-6.188286164277605}
\parametricplot{-0.5451639760458047}{1.026761915674777}{1.*2.11023076856285*cos(t)+0.*2.11023076856285*sin(t)+19.061764114132604|0.*2.11023076856285*cos(t)+1.*2.11023076856285*sin(t)+-7.1128060359531355}
\parametricplot{4.641409953313899}{6.263799006984062}{1.*2.1641823373559235*cos(t)+0.*2.1641823373559235*sin(t)+18.702324953860114|0.*2.1641823373559235*cos(t)+1.*2.1641823373559235*sin(t)+-8.165130863684741}
\parametricplot{3.461208458417178}{5.146713400951539}{1.*2.2631563520937314*cos(t)+0.*2.2631563520937314*sin(t)+17.596511490175565|0.*2.2631563520937314*cos(t)+1.*2.2631563520937314*sin(t)+-8.270831349653568}
\parametricplot{2.407330004843081}{4.020494328096269}{1.*1.9625913324215616*cos(t)+0.*1.9625913324215616*sin(t)+16.700097904542798|0.*1.9625913324215616*cos(t)+1.*1.9625913324215616*sin(t)+-7.470648141619793}
\parametricplot{0.7976244612339359}{1.9468396888227306}{1.*8.118131770767691*cos(t)+0.*8.118131770767691*sin(t)+9.750418363318909|0.*8.118131770767691*cos(t)+1.*8.118131770767691*sin(t)+-0.788066913445526}
\parametricplot{1.0430376122033562}{1.818043032674664}{1.*8.030127583522964*cos(t)+0.*8.030127583522964*sin(t)+11.376235021755795|0.*8.030127583522964*cos(t)+1.*8.030127583522964*sin(t)+-1.9154599433627963}
\parametricplot{1.4563827320567855}{2.4241903030727996}{1.*6.1049183639277*cos(t)+0.*6.1049183639277*sin(t)+14.723230695329606|0.*6.1049183639277*cos(t)+1.*6.1049183639277*sin(t)+-1.0429318116957487}
\parametricplot{0.8433906634939082}{2.298201990095885}{1.*8.211339304272135*cos(t)+0.*8.211339304272135*sin(t)+12.229089788539724|0.*8.211339304272135*cos(t)+1.*8.211339304272135*sin(t)+0.6297445198973521}
\parametricplot{0.5307002462138974}{2.4794867013759823}{1.*8.213591570905459*cos(t)+0.*8.213591570905459*sin(t)+13.247143594687728|0.*8.213591570905459*cos(t)+1.*8.213591570905459*sin(t)+1.713265013968647}
\parametricplot{2.3844334399537823}{3.8670885941781896}{1.*8.27753686455291*cos(t)+0.*8.27753686455291*sin(t)+10.516225789310285|0.*8.27753686455291*cos(t)+1.*8.27753686455291*sin(t)+-0.6634333650936594}
\parametricplot{2.040734916575027}{4.246417331413974}{1.*7.845454416401478*cos(t)+0.*7.845454416401478*sin(t)+8.052862581801131|0.*7.845454416401478*cos(t)+1.*7.845454416401478*sin(t)+-1.9729045524532125}
\parametricplot{2.4620605988709703}{4.257415590158902}{1.*6.181277372614647*cos(t)+0.*6.181277372614647*sin(t)+9.308404339628144|0.*6.181277372614647*cos(t)+1.*6.181277372614647*sin(t)+1.1375774578768725}
\parametricplot{2.9881160215902125}{3.8512214157953197}{1.*8.217876248766775*cos(t)+0.*8.217876248766775*sin(t)+12.826225856957802|0.*8.217876248766775*cos(t)+1.*8.217876248766775*sin(t)+0.9394799038125918}
\parametricplot{2.3204404368642146}{3.5099248191424763}{1.*4.824888827067628*cos(t)+0.*4.824888827067628*sin(t)+11.093393128656995|0.*4.824888827067628*cos(t)+1.*4.824888827067628*sin(t)+-2.677646104062557}
\parametricplot{3.6428046472793576}{5.299441599221046}{1.*6.012075891403089*cos(t)+0.*6.012075891403089*sin(t)+15.218696330725306|0.*6.012075891403089*cos(t)+1.*6.012075891403089*sin(t)+-5.318348528992833}
\parametricplot{3.8901302171207837}{5.2548253663098725}{1.*4.405340061198156*cos(t)+0.*4.405340061198156*sin(t)+13.173826939149718|0.*4.405340061198156*cos(t)+1.*4.405340061198156*sin(t)+-5.208950324305381}
\parametricplot{4.797066384515911}{5.36669195779326}{1.*10.10842105226708*cos(t)+0.*10.10842105226708*sin(t)+9.09116830541027|0.*10.10842105226708*cos(t)+1.*10.10842105226708*sin(t)+1.8651189595353237}
\parametricplot{4.002215080046307}{5.422562880723072}{1.*8.374674461956047*cos(t)+0.*8.374674461956047*sin(t)+13.088842347738016|0.*8.374674461956047*cos(t)+1.*8.374674461956047*sin(t)+-3.973779454055089}
\parametricplot{3.555852002800916}{5.678086279922869}{1.*8.067893065505206*cos(t)+0.*8.067893065505206*sin(t)+11.913437817945141|0.*8.067893065505206*cos(t)+1.*8.067893065505206*sin(t)+-5.734495170627755}
\parametricplot{-1.0090198094237097}{1.0850080015374755}{1.*8.134412428102484*cos(t)+0.*8.134412428102484*sin(t)+16.53297621787999|0.*8.134412428102484*cos(t)+1.*8.134412428102484*sin(t)+-1.3228482227501959}
\parametricplot{-0.7361280482205004}{0.7044647751728856}{1.*8.474059674156793*cos(t)+0.*8.474059674156793*sin(t)+14.586196790761457|0.*8.474059674156793*cos(t)+1.*8.474059674156793*sin(t)+-2.5174004977123197}
\parametricplot{-0.5821813329004293}{1.046098183750467}{1.*6.402021972357773*cos(t)+0.*6.402021972357773*sin(t)+15.518714215488176|0.*6.402021972357773*cos(t)+1.*6.402021972357773*sin(t)+-4.686949644049169}
\parametricplot{-0.32224315207225107}{0.7778121946936477}{1.*6.049656986099362*cos(t)+0.*6.049656986099362*sin(t)+14.415737517693389|0.*6.049656986099362*cos(t)+1.*6.049656986099362*sin(t)+-3.3913382181890865}
\parametricplot{-0.883093138061092}{0.4302730868881969}{1.*4.428161267824218*cos(t)+0.*4.428161267824218*sin(t)+14.701275838113478|0.*4.428161267824218*cos(t)+1.*4.428161267824218*sin(t)+-0.9932286771151639}
\rput[tl](6.621081890881592,4.700485818484407){\fontsize{33pt}{33pt}$G_1$}
\rput[tl](17.52489960235613,4.771060690403336){\fontsize{33pt}{33pt}$G_2$}
\rput[tl](17.4543247304372,-6.379769072787405){\fontsize{33pt}{33pt}$G_3$}
\rput[tl](6.409357275124805,-6.2386193289495475){\fontsize{33pt}{33pt}$G_4$}
\begin{scriptsize}
\psdots[dotstyle=*,linecolor=xdxdff](8.107347062658953,6.6718824963596415)
\psdots[dotstyle=*,linecolor=xdxdff](9.41097914573562,5.8704708519515645)
\psdots[dotstyle=*,linecolor=xdxdff](10.107337026165851,4.682257546512569)
\psdots[dotstyle=*,linecolor=xdxdff](10.123076779374983,2.970621748398977)
\psdots[dotstyle=*,linecolor=xdxdff](9.294329756949,1.613740809118493)
\psdots[dotstyle=*,linecolor=xdxdff](7.805818504007426,0.8538415825820227)
\psdots[dotstyle=*,linecolor=xdxdff](6.035326486199539,1.0524990036129829)
\psdots[dotstyle=*,linecolor=xdxdff](4.704946125277688,2.195786219902287)
\psdots[dotstyle=*,linecolor=xdxdff](4.246557098240255,3.601758748227831)
\psdots[dotstyle=*,linecolor=xdxdff](4.500193430394818,5.0220719950756925)
\psdots[dotstyle=*,linecolor=xdxdff](5.408985574061948,6.176422978345132)
\psdots[dotstyle=*,linecolor=xdxdff](6.769089788539724,6.762810080437562)
\psdots[dotstyle=*,linecolor=xdxdff](19.027347062658954,6.671882496359641)
\psdots[dotstyle=*,linecolor=xdxdff](20.33097914573562,5.870470851951566)
\psdots[dotstyle=*,linecolor=xdxdff](21.02733702616585,4.682257546512569)
\psdots[dotstyle=*,linecolor=xdxdff](21.043076779374985,2.970621748398977)
\psdots[dotstyle=*,linecolor=xdxdff](20.214329756949002,1.6137408091184948)
\psdots[dotstyle=*,linecolor=xdxdff](18.725818504007425,0.8538415825820227)
\psdots[dotstyle=*,linecolor=xdxdff](16.95532648619954,1.0524990036129829)
\psdots[dotstyle=*,linecolor=xdxdff](15.624946125277688,2.195786219902287)
\psdots[dotstyle=*,linecolor=xdxdff](15.166557098240254,3.601758748227831)
\psdots[dotstyle=*,linecolor=xdxdff](15.420193430394818,5.0220719950756925)
\psdots[dotstyle=*,linecolor=xdxdff](16.32898557406195,6.176422978345132)
\psdots[dotstyle=*,linecolor=xdxdff](17.689089788539725,6.762810080437562)
\psdots[dotstyle=*,linecolor=xdxdff](7.930370906389543,-4.505822976230966)
\psdots[dotstyle=*,linecolor=xdxdff](9.23400298946621,-5.307234620639043)
\psdots[dotstyle=*,linecolor=xdxdff](9.930360869896441,-6.495447926078038)
\psdots[dotstyle=*,linecolor=xdxdff](9.946100623105576,-8.20708372419163)
\psdots[dotstyle=*,linecolor=xdxdff](9.117353600679591,-9.563964663472113)
\psdots[dotstyle=*,linecolor=xdxdff](7.628842347738017,-10.323863890008585)
\psdots[dotstyle=*,linecolor=xdxdff](5.858350329930129,-10.125206468977623)
\psdots[dotstyle=*,linecolor=xdxdff](4.527969969008279,-8.98191925268832)
\psdots[dotstyle=*,linecolor=xdxdff](4.0695809419708455,-7.575946724362776)
\psdots[dotstyle=*,linecolor=xdxdff](4.323217274125408,-6.155633477514915)
\psdots[dotstyle=*,linecolor=xdxdff](5.232009417792538,-5.001282494245476)
\psdots[dotstyle=*,linecolor=xdxdff](6.592113632270315,-4.414895392153045)
\psdots[dotstyle=*,linecolor=xdxdff](18.850370906389543,-4.5058229762309665)
\psdots[dotstyle=*,linecolor=xdxdff](20.154002989466207,-5.307234620639041)
\psdots[dotstyle=*,linecolor=xdxdff](20.850360869896438,-6.495447926078038)
\psdots[dotstyle=*,linecolor=xdxdff](20.866100623105574,-8.20708372419163)
\psdots[dotstyle=*,linecolor=xdxdff](20.03735360067959,-9.563964663472113)
\psdots[dotstyle=*,linecolor=xdxdff](18.548842347738013,-10.323863890008585)
\psdots[dotstyle=*,linecolor=xdxdff](16.778350329930127,-10.125206468977623)
\psdots[dotstyle=*,linecolor=xdxdff](15.447969969008277,-8.98191925268832)
\psdots[dotstyle=*,linecolor=xdxdff](14.989580941970843,-7.575946724362776)
\psdots[dotstyle=*,linecolor=xdxdff](15.243217274125406,-6.155633477514915)
\psdots[dotstyle=*,linecolor=xdxdff](16.152009417792538,-5.001282494245475)
\psdots[dotstyle=*,linecolor=xdxdff](17.512113632270314,-4.414895392153044)
\end{scriptsize}
\end{pspicture*}
}

\caption{A U$3$LC planar graph without $K_4$ as a subgraph}
\label{planar2-fig}
\centering
\end{figure}

In the graph shown in figure~\ref{planar2-fig}, we assign list of colors $\{1, 2, 3, 4\} \setminus \{ i\}$ to all of the vertices of $G_i$ $(1\leq i\leq 4)$. The proof of uniqueness of coloring is similar to  Lemma~\ref{tight}.

With the same average-degree method as in Corollary~\ref{cor1}, which showed every planar graph $G$ has property $M(4)$, it is easy to deduce that graphs on nonplanar surfaces have property $M(k)$ for some $k$. If a graph $G$ has a $2$-cell embedding on a surface of Euler genus $g \geq 1$, it satisfies the Euler-Poincar{\'e} formula $|V(G)|-|E(G)|+|F(G)|=2-g$ so that its average degree is bounded by $2\frac{|E(G)|}{|V(G)|} \leq 6+6\frac{(g-2)}{|V(G)|}$. For $g\geq 1$, define the Heawood number to be $H(g)=\lfloor \frac{(7+\sqrt{24g+1})}{2}\rfloor$. It is known that every graph that embeds on a surface of Euler genus $g\geq 1$ has chromatic number at most $H(g)$ and also has list chromatic number at most $H(g)$~\cite[pages 3, 20]{MR1304254}. Just as for the plane, average degree arguments show that for a graph $G$ on a surface of Euler genus $g\geq 1$, $G$ has property $M(H(g)-1)$, but more is true as shown in the next theorem.

\begin{theorem}
Given integer $i\geq 2$, if $g\geq \frac{(2i^2-3i+1)}{3}$ and if a graph $G$ embeds on a surface of Euler genus $g$, then $G$ has property $M(H(g)-i)$. Equivalently given $g\geq 1$, if $i \leq \frac{(3+\sqrt{24g+1})}{4}$ and if $G$ embeds on a surface of Euler genus $g$, then $G$ satisfies property $M(H(g)-i)$.
\end{theorem}
\begin{proof}
Suppose $G$ with $n =|V(G)|$  vertices embeds on a surface of Euler genus $g \geq 1$. From Euler's formula its average degree is bounded by $\frac{2|E(G)|}{n}\leq 6+ \frac{6(g-2)}{n}$. We consider two cases.

Suppose $n<\frac{(7+\sqrt{24g+1})}{2}$ so that $n\leq H(g)$. If $G=K_{H(g)}$, then $m(G)=2<H(g)-i$ by {\rm~\cite{MR1625324}}. Otherwise $G$ is a proper subgraph of $K_{H(g)}$ and its average degree is strictly less than $H(g)-1$. We have $H(g)-1 \leq 2H(g) - 2i -2$, so that $m(G) \leq H(g) - i$ by Theorem~\ref{main_result}.

Next suppose $n\geq \frac{(7+\sqrt{24g+1})}{2}$. Then:
\begin{equation*}
\begin{split}
&\frac{2|E(G)|}{n}\leq 6+\frac{6(g-2)}{n}\\
&\leq 6+\frac{12(g-2)}{(7+\sqrt{24g+1})}\\
&=\frac{(5+\sqrt{24g+1})}{2}\\
&<H(g)\leq 2H(g)-2i-2.
\end{split}
\end{equation*}

Thus we have again that $m(G)\leq H(g)-i$ by Theorem~\ref{main_result}.	
\qed
\end{proof}

We do not have examples to show that these bounds are best possible.

\section{Bounds on $m(G)$ for $k$-regular graphs $G$}
\label{bounds-sec}

\begin{question}
What is the best upper bound for the $m-number$ of a $k$-regular graph?
\end {question}

\begin{question}
What is the best upper bound for the $m-number$ for all but finitely many $k$-regular graphs?
\end {question}

In the following we answer these questions for $k=3$, we investigate them in depth for $k=4$, and we show that $m(G) \le k$ for all $k$-regular graphs $G$.

\begin{theorem} 
Except for $K_{3,3}$ and $K_4$, the $m-number$ of every connected 3-regular graph is 3.
\end{theorem}
\begin{proof}
Let $G$ be a 3-regular connected graph. If $G$ is a 2-connected graph then by Theorem~\ref{m2}, $G$ has property $M(2)$ if and only if it is a $K_{3,3}$ or $K_4$, so $m(G)\geq 3$. Otherwise there is a cut-vertex $v$ of $G$ in an end-block $B$ of $G$.  
By definition all vertices of $B$ have degree $3$ except for $v$, so $B$ does not satisfy the conditions of Theorem~\ref{m2} and $m(G) \geq 3$. Also according to Theorem~\ref{main_result}, $m(G) \leq 3$  and the proof is complete.
\qed
\end{proof}

\begin{definition}
We define a 
\begin{sf}
Strip Triangle $ST_n$
\end{sf}
to be a graph that is obtained by adding a new vertex, say $t$, for each pair of adjacent vertices $(u,v) \in P_{n+1}$ and replacing the edge $(u,v)$ by a triangle on $\{u,t,v\}$.
\end{definition}

\begin{theorem}
There exist infinitely many 4-regular graphs $G$ with $m(G)=3$.
\end{theorem}
\begin{proof}
Consider $H_n$ as a graph which is obtained by extending $ST_n$ by "hanging" a $K_5$ on each degree-two vertex $t \in V(ST_n)$. Hanging a $K_5$ on a vertex $t$ means replacing an edge $(u,v)$ of $K_5$ with the path $(u,t,v)$. So $H_n$ is a connected 4-regular graph. We prove that $m(H_n)=3$: We know that $H_n$ contains an induced subgraph $K_4$, so by Theorem~\ref{induced}  it follows that $H_n$ has property $M(3)$.
\qed
\end{proof}

The following corollary is immediate from Theorem~\ref{main_result}.

\begin{corollary}(\cite{MR1899929})
Every $k$-regular graph $G$ $(k\geq 2)$ has property $M(\lceil \frac{k+1}{2}\rceil + 1)$.
\end{corollary}

\begin{corollary}
For every graph $G$ with average degree less than $4$, the conjecture $\chi_u(G) \leq \Delta + 1$ in~\cite{MR1906860}, holds.
\end{corollary}
\begin{proof}
By Theorem~\ref{main_result} each graph $G$ with the above condition has property $M(3)$, and by Theorem~\ref{chi_theorem} the result is clear.
\qed
\end{proof}

\begin{theorem}
In every $k$-regular graph $G$ with a list of $k$ colors on each of its vertices, if there is at least one list coloring for the graph, then there are at least $n(\frac{k}{4}-1)+1$ list colorings.
\end{theorem}
\begin{proof}
If there is at least one list coloring for the graph $G$, we can run directing procedure on $G$. Hence, after Step $1$ of the directing procedure, suppose that $a_i$ denotes the number of colors that there are in $L(v_i)$ but none of the neighbors of $v_i$ got these colors. There are at least $k - 2a_i$ edges directed from $v_i$ to its neighbors. Therefore, the number of bidirectional edges of the graph is at least $max\{0,(\sum k-2a_i)-\frac{kn}{2}\}$. Because for each color in $L(v_i)$ that counted in $a_i$ and also for each bidirectional edge in the graph there is a new list coloring, the number of list colorings is at least:
\begin{equation*}
\begin{split}
&1 + \sum (a_i-1) + max\{0,\sum (k-2a_i)-\frac{kn}{2}\} \\
	&= 1 - n + max\{\sum a_i,\ \frac{kn}{2}-\sum a_i\}\\
	&\geq n(\frac{k}{4}-1)+1.
\end{split}
\end{equation*}\qed
\end{proof}
\begin{corollary}
For $k \geq 4$, in every $k$-regular graph $G$ with a list of $k$ colors on each of its vertices, then there are at least $n(\frac{k}{4}-1)+1$ list colorings and $G$ has property $M(k)$ unless $G = K_{k+1}$  with all $k$-lists identical.
\end{corollary}

\begin{remark}
$\overline{C_6}$ is the smallest (i.e., fewest vertices and fewest edges) U2LC 3-regular graph since the only other 3-regular graph with 6 vertices is $K_{3,3}$.
\end{remark}

\begin{question}
Are there infinitely many 4-regular graphs with $m-number = 4$?
\end{question}

\begin{question}
Except for $K_5$ and $K_{4,4}$ does every 4-regular graph have $m-number = 3$?
\end{question}

\section{Uniquely list colorable graphs with lists of varying sizes}
\label{varylist-sec}

Until now, only U$k$LC graphs have been discussed. We can extend these results to {\sf uniquely list colorable} graphs where the lists of colors on the vertices may have different sizes. The following theorem is a generalization from Theorem~\ref{main_result}.

For the next two theorems suppose $k_i$ is the size of the list of vertex $v_i$ of degree $d_i$.

\begin{theorem}
Let $G$ be a graph with $n$ vertices and a list assignment $L$. If there exists a unique proper $L$-list coloring on $G$, then:
\begin{equation*}
\sum_{i=1}^{n} (d_i-2(k_i-1)) \geq 0.
\end{equation*}
So the average degree of $G$ is at least $\frac{\sum_{i=1}^{n} (2(k_i-1))}{n}$.
\end{theorem}
\begin{proof}
The proof is similar to Theorem~\ref{main_result}. 
\qed
\end{proof}

The next theorem suggests a test for cases when equality holds in Theorem 6. Then Theorem 7 could be used to identify some non-uniquely list colorable graphs.

\begin{theorem}
Suppose that $G$ is a uniquely list colorable graph with list assignment $L$, with $n$ vertices, and average degree $\frac{\sum_{i=1}^{n} (2(k_i-1))}{n}\cdot$ Then for each color in the union of members of $L$, say $t$, the following condition  holds:
\begin{equation*}
\sum_{v_i\in V(t)} (d_i-2(k_i-1))=2b_t-d_t.
\end{equation*}
In this inequality, $V(t)$ denotes the set of vertices with color $t$, $\widetilde{V}(t)$ is $V(G)\setminus V(t)$, and $b_t$ denotes the number of vertices in $\widetilde{V}(t)$ with at least one neighbor in $V(t)$ and containing $t$ in its list. Also, $d_t$ denotes the number of edges between $V(t)$ and $\widetilde{V}(t)$.
\end{theorem}
\begin{proof}
Suppose $C= \bigcup_{v_i\in V(G)} L(v_i)$ denotes the set of all colors that exist in the union of members of $L$. For each color $t \in C$, we put aside $V(t)$ and remove $t$ from lists of vertices in $\widetilde{V}(t)$. (It is clear that there is no vertex in $\widetilde{V}(t)$ with $t$ in its list and has no neighbor in $V(t)$.) Then if we write the equation in Theorem~\ref{main_result} for the remained graph, the following equation results:
\begin{equation*}
\sum_{i=1}^{n} (d_i-2(k_i-1)) - 	\sum_{v_i\in V(t)} (d_i-2(k_i-1)) + 2b_t-d_t 
\end{equation*}
\begin{equation*}
= -\sum_{v_i\in V(t)} (d_i-2(k_i-1)) + 2b_t-d_t \geq 0.
\end{equation*}
Now for variable $d_t$ we have the following equation because each edge is counted twice: 
\begin{equation*}
	\sum_{t\in C} d_t = 2|E(G)|
\end{equation*}
and finally for $b_t$ the following equation results since vertex $v_i$ is counted $k_{i}-1$ times in $\sum b_t$:
\begin{equation*}
	\sum_{t\in C} b_t = |E(G)|
\end{equation*}

If we sum up the above equation for each color $t \in C$, the claims will be proved.
\qed
\end{proof}


\section{Further problems}
\label{further-sec}
\begin{conjecture}
All graphs with average degree $2k-2$ have property $M(k)$ (changing ``less than'' to ``equal'' in Theorem~\ref{main_result}).
\end{conjecture}
\begin{acknowledgements}
Part of the research of E. S. M. was supported by INSF and 
the Research Office of the Sharif University of Technology.
\end{acknowledgements}

\end{document}